\documentclass[11pt,a4paper]{amsart}

\usepackage[utf8]{inputenc}
\usepackage[T1]{fontenc}
\usepackage{amsmath,amssymb,amsthm,mathtools}
\usepackage[margin=2.8cm]{geometry}
\usepackage{enumitem}
\usepackage{graphicx}
\usepackage{tikz}
\usetikzlibrary{arrows.meta,decorations.markings,decorations.pathreplacing,calc}

\newtheorem{theorem}{Theorem}[section]
\newtheorem{proposition}[theorem]{Proposition}
\newtheorem{lemma}[theorem]{Lemma}

\theoremstyle{definition}
\newtheorem{definition}[theorem]{Definition}
\newtheorem{remark}[theorem]{Remark}

\newcommand{\R}{\mathbb{R}}
\newcommand{\Z}{\mathbb{Z}}
\newcommand{\N}{\mathbb{N}}
\newcommand{\C}{\mathbb{C}}
\newcommand{\dd}{\,\mathrm{d}}
\newcommand{\G}{\mathbb{G}}
\newcommand{\F}{\mathbb{F}}
\newcommand{\EE}{\mathcal{E}}

\newcommand{\vA}{\vec A}
\newcommand{\vB}{\vec B}
\newcommand{\vC}{\vec C}
\newcommand{\argch}{\operatorname{argch}}
\newcommand{\dist}{\operatorname{dist}}

\title{The Green function for waves on the $2$-regular Bethe lattice}

\author{Kaïs Ammari}
\address{LR Analysis and Control of PDEs, LR22ES03, Department of Mathematics, Faculty of Sciences of Monastir, University of Monastir, 5019 Monastir, Tunisia}
\email{kais.ammari@fsm.rnu.tn}

\author{Gilles Lebeau}
\address{Université Côte d'Azur, CNRS, Laboratoire J.-A. Dieudonné, Parc Valrose, 06108 Nice Cedex 02, France}
\email{gilles.lebeau@univ-cotedazur.fr}

\thanks{Gilles Lebeau was supported by the European Research Council, ERC-2012-ADG, project number 320845: Semi Classical Analysis of Partial Differential Equations.}

\begin{document}

\begin{abstract}
We compute an explicit analytic expression for the Green function of the wave operator on the $2$-regular lattice called the ``Bethe lattice'', equipped with its standard metric. In particular, we exhibit a phenomenon of abnormal speed of propagation for waves: the effective speed of propagation of energy for large time is $c_*=2\sqrt2/3<1$, and there is a true propagation at any speed $c<c_*$. We also show that the propagator at integer times has two $\ell^2$ eigenvectors, with eigenvalues $\pm1$. They are standing waves whose amplitude is halved at each generation, and they keep one third of the energy of an initial Dirac mass trapped near the source for all times.
\end{abstract}

\maketitle

\tableofcontents

\section{Introduction}

A $z$-regular Bethe lattice (a particular kind of Cayley graph, introduced by Hans Bethe in 1935) is an infinite connected cycle-free graph in which each node is connected to $z+1$ neighbours. The $z$-regular Bethe lattice $\mathcal B_z$ is thus a one-dimensional Riemannian manifold with singularities (the nodes). The canonical metric is the obvious one: each edge is identified with $]0,1[$ and has length $1$. Moreover, given two points $p,q$ in $\mathcal B_z$, there exists a unique geodesic connecting $p$ and $q$. The number of nodes at a given distance $k\in\N^*$ of a given node is equal to
\[
N_k=(z+1)z^{k-1}.
\]
Due to its distinctive topological structure, statistical mechanics models on such graphs are often exactly solvable (see \cite{Baxter}). In this paper we study the wave equation on $\mathcal B_z$ with speed $1$, and we show that it is effectively exactly solvable. To our knowledge, our explicit formulas are new.

Let $\mathcal N_z$ be the set of nodes and $\mathcal A_z$ the set of edges. Then $\mathcal B_z\setminus\mathcal N_z$ is the disjoint union of the edges $A\in\mathcal A_z$, and each edge $A$ is identified with the interval $]0,1[$. By definition, a wave on $\mathcal B_z$ with speed $1$ is a collection of distributions $(u_A(x,t))_{A\in\mathcal A_z}$, defined on $]0,1[\times\R_t$, which satisfy the usual one-dimensional wave equation on each edge,
\begin{equation}\label{1.1}
\frac{\partial^2u_A}{\partial t^2}(x,t)-\frac{\partial^2u_A}{\partial x^2}(x,t)=0,\qquad 0<x<1,\ t\in\R,\ A\in\mathcal A_z,
\end{equation}
together with the Kirchhoff boundary conditions at the nodes $N\in\mathcal N_z$ (see \cite{ACFN,AMN,BI,KS,Nicaise}). If $A_i$ denotes the set of edges adjacent to $N$, this means
\begin{align}
u_{A_i}(x=N,t)&=u_{A_j}(x=N,t),\qquad\forall i,j,\label{1.2}\\
\sum_i\partial_nu_{A_i}(x=N,t)&=0.\label{1.3}
\end{align}
Since the distribution $u_A(x,t)$ satisfies the wave equation \eqref{1.1}, the traces $u_A(x=N,t)$ and $\partial_nu_A(x=N,t)$ are well-defined distributions on $\R_t$. Therefore the boundary conditions \eqref{1.2} and \eqref{1.3} are well defined. It is well known that the Cauchy problem associated with \eqref{1.1}, \eqref{1.2} and \eqref{1.3} is well posed.

For simplicity, throughout the paper we restrict ourselves to the case $z=2$, and we set $\mathcal B=\mathcal B_2$, $\mathcal N=\mathcal N_2$ and $\mathcal A=\mathcal A_2$.

\begin{figure}[ht]
\centering
\begin{tikzpicture}[scale=1.05,every node/.style={inner sep=0pt}]
\tikzset{dot/.style={circle,fill,inner sep=1.1pt}}
\coordinate (L) at (-0.8,0); \coordinate (Rr) at (0.8,0);
\draw[thick] (L)--(Rr); \node[dot] at (L){}; \node[dot] at (Rr){}; \node[dot] at (0,0){};
\node at (0,0.22) {$A$};
\foreach \s/\P in {1/Rr,-1/L}{
 \foreach \a in {35,-35}{
  \coordinate (B) at ($(\P)+(\s*1.3,{1.3*tan(\a)})$);
  \draw (\P)--(B); \node[dot] at (B){};
  \foreach \b in {20,-20}{
   \coordinate (C) at ($(B)+({\s*1.1*cos(\a+\b)},{1.1*sin(\a+\b)})$);
   \draw (B)--(C); \node[dot] at (C){};
   \foreach \c in {18,-18}{
    \draw[dashed] (C)--++({\s*0.55*cos(\a+\b+\c)},{0.55*sin(\a+\b+\c)});
   }
  }
 }
}
\end{tikzpicture}
\caption{The Bethe lattice $\mathcal B$.}
\label{fig1}
\end{figure}

Each edge $A$ has two orientations and a middle point $m(A)$, defined by $x=1/2$ in the identification $A=]0,1[$. If $\vA$ is a given orientation, we denote by $-\vA$ the opposite orientation. We write $\vA\to\vB$ if the right end point of $\vA$ is equal to the left end point of $\vB$.

As in Figure~\ref{fig1}, we fix an edge $A$, oriented from left to right. We denote by $W(t,p,\vA)$ the unique solution of \eqref{1.1}, \eqref{1.2} and \eqref{1.3} such that
\begin{equation}\label{1.4}
\begin{cases}
u_A(x,t)=\delta_{x=1/2+t}, & \forall t\in\,]-1/2,1/2[,\\
u_B(x,t)=0, & \forall t\in\,]-1/2,1/2[,\ \forall B\neq A.
\end{cases}
\end{equation}
By finite speed of propagation and elementary properties of one-dimensional waves, the knowledge of $W(t,p,\vA)$ allows one to solve the Cauchy problem associated with \eqref{1.1}, \eqref{1.2} and \eqref{1.3} for any Cauchy data. Moreover, $W(t,p,\vA)$ is at any time a finite sum of Dirac masses (see Section~\ref{sec2}). For $t\in\Z$, these Dirac masses are located at the middle points of the edges, and they propagate according to one of the two orientations on each edge. Therefore, for any $k\in\Z$,
\begin{equation}\label{1.5}
W(t,\cdot,\vA)=\sum_{\vB,\,B\in\mathcal A}G(k,\vB,\vA)\,\delta_{x=1/2\pm(t-k)},\qquad\forall t\in\,]k-1/2,k+1/2[,
\end{equation}
where the sign $\pm$ depends on the orientation of the edge $B$.

The aim of this paper is to give an explicit formula for the coefficients $G(k,\vB,\vA)$ with $k\in\N$; the case $k\in-\N$ follows by time symmetry and reversal of orientation. This is done in Section~\ref{sec3}, see Proposition~\ref{prop3.10}. An important quantity, related to the propagation of energy, is $\EE(k,d)$, $k,d\in\N$, defined by
\begin{equation}\label{1.6}
\EE(k,d)=\sum_{\vA\ \text{fixed},\ \dist(m(B),m(A))=d}|G(k,\vB,\vA)|^2 .
\end{equation}
By finite speed of propagation, one has obviously $\EE(k,d)=0$ for all $d>k$. Since the propagator is unitary (Lemma~\ref{lem3.1}), $\sum_d\EE(k,d)=1$ for every $k$. The main result of this paper is the following theorem.

\begin{theorem}\label{thm1.1}
Let $c_*=2\sqrt2/3<1$. For $k\in\N^*$ and $d\in\N$, let $\gamma=d/k$.
\begin{enumerate}[label=(\arabic*),leftmargin=*]
\item For $\gamma>c_*$, there exists $a=a(\gamma)>0$ such that $\EE(k,d)\in O(e^{-ak})$.
\item For $\gamma\in\,]0,c_*[$, there exist $b_{j,m}=b_{j,m}(\gamma)$ and $a_{j,m}=a_{j,m}(\gamma)$, $j,m\in\{1,2,3,4\}$, such that
\[
\EE(k,d)=k^{-1}\sum_j\Big|\sum_me^{ia_{j,m}k}\big(b_{j,m}+O(k^{-1})\big)\Big|^2 .
\]
\item One has $G(k,-\vA,\vA)=-\frac13$ for every odd $k$ and $G(k,-\vA,\vA)=0$ for every even $k$. Moreover $G(k,\vA,\vA)=\frac13\mathbf 1_{k\ \mathrm{even}}+O(k^{-3/2})$. In particular,
\[
\EE(k,0)=\frac19+O(k^{-3/2}),\qquad \EE(k,0)=\frac19\ \text{ for every odd }k .
\]
\item For every fixed $d\ge1$, $\EE(k,d)=\frac29\,2^{-d}+O_d(k^{-3/2})$ as $k\to\infty$.
\item The energy splits as follows: a fraction $\frac13$ of it remains trapped near $m(A)$ for all times,
\[
\lim_{d_0\to\infty}\lim_{k\to\infty}\sum_{d\le d_0}\EE(k,d)=\frac19+\sum_{d\ge1}\frac29\,2^{-d}=\frac13,
\]
while the remaining fraction $\frac23$ propagates with speeds in $]0,c_*]$.
\end{enumerate}
\end{theorem}

This theorem is about the behaviour of $\EE(k,d)$ for large values of $k$ (large time). Items (1) and (2) are obtained in Section~\ref{sec4} from the study of a phase integral with large parameter $k$. The precise statement for $d\simeq c_*k$ is given in Theorem~\ref{thm4.3}; it involves a transition described by Airy functions. Items (3), (4) and (5) are proved in Section~\ref{sec5}. They come from two simple poles of the generating functions on the boundary of their disc of convergence, which correspond to two $\ell^2$ eigenvectors of the propagator at integer times.

\begin{remark}\label{rem1.2}
We thus find an effective speed of propagation of energy which is at most
\[
c_*=2\sqrt2/3<1.
\]
Moreover, any speed $c\in\,]0,c_*]$ appears in the propagation of the energy. This is a kind of ``diffusion'', but it is quite different from heat or parabolic diffusion: there is no regularization at all, and the scaling factor between space and time is the same.

Finally, one third of the energy does not propagate at all. It is carried by two $\ell^2$ standing waves, which vanish at the nodes and whose amplitude is halved at each generation (Proposition~\ref{prop5.2}). This is possible on the tree because the number of edges only doubles at each generation, so that these standing waves have finite energy.
\end{remark}

\begin{figure}[ht]
\centering
\includegraphics[width=\textwidth]{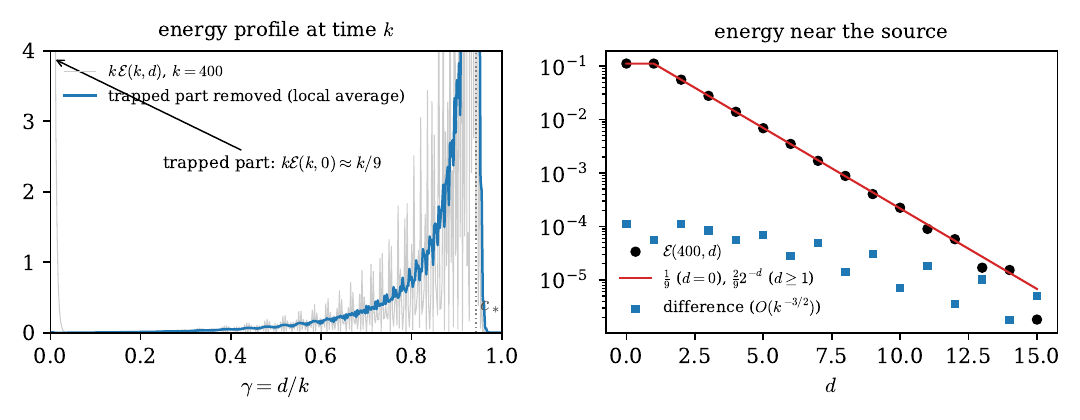}
\caption{The energy at time $k$ as a function of the distance $d$, for $k=400$. Left: $k\,\EE(k,d)$ against $\gamma=d/k$. The grey curve is the exact value; the blue curve is a local average after removing the trapped part. The propagating part is of order $1/k$ up to the front $\gamma=c_*$, beyond which it is exponentially small. The trapped part produces the peak $k\EE(k,0)\approx k/9$ at $\gamma=0$. Right: $\EE(400,d)$ for small $d$ (dots), compared with the trapped profile $\frac19$ ($d=0$) and $\frac29 2^{-d}$ ($d\ge1$) of Theorem~\ref{thm1.1}(3)--(4). The difference (squares) is of order $k^{-3/2}$.}
\label{fig2}
\end{figure}

\subsection*{Note on this version}
In the first version of this paper (arXiv:1711.01363v1), item (3) of Theorem~\ref{thm1.1} stated that $\EE(k,0)\in O(k^{-3})$, and Lemma~\ref{lem4.1} stated that the generating functions $g_j$ are analytic on $\partial D_r$ except at four branch points. Both statements were incorrect: the $g_j$ also have simple poles at $x=\pm r$, where their common denominator vanishes. The elimination leading to (4.37) in v1 was erroneous. In the present version:
\begin{itemize}
\item Lemma~\ref{lem4.1} and its proof are corrected;
\item items (3)--(5) of Theorem~\ref{thm1.1}, Remark~\ref{rem1.2} and Figure~\ref{fig2} are corrected or new;
\item Section~\ref{sec5} is new, and describes the poles and the corresponding trapped modes;
\item the proof of Theorem~\ref{thm4.3} is completed to take the poles into account.
\end{itemize}
Items (1) and (2) of Theorem~\ref{thm1.1} and the statements of Theorem~\ref{thm4.3} are unchanged. Formula \eqref{3.32} for $G(k,\vB,\vA)$ is unchanged, and it has been checked against a direct computation of the matrix powers $\G(1)^k$.

\section{The propagation of waves at nodes}\label{sec2}

In this section we compute the reflection and transmission coefficients for a Dirac mass at a given node.

\begin{figure}[ht]
\centering
\begin{tikzpicture}[scale=1.2,>=Stealth]
\draw (-2,0)--(0,0); \draw (0,0)--(1.4,1); \draw (0,0)--(1.4,-1);
\draw[->] (-1.4,0.12)--(-0.6,0.12);
\draw[->] (-0.5,0.25) to[bend right=40] (-1.5,0.25);
\draw[->] (0.35,0.42)--(0.95,0.85);
\draw[->] (0.35,-0.42)--(0.95,-0.85);
\node at (-1,-0.25) {\small $\delta_{x=t}$};
\node at (-1,0.75) {\small $\alpha$};
\node at (0.45,0.85) {\small $\beta$};
\node at (0.45,-0.85) {\small $\beta$};
\end{tikzpicture}
\caption{Reflection and transmission at a node.}
\label{fig3}
\end{figure}

As in Figure~\ref{fig3}, consider an incoming wave on the left edge, identified with $]-1,0[$, of the form $u_{\rm in}(x,t)=\delta_{x=t}$ for small $t<0$. We identify the two right edges with $]0,1[$. Then, for small $t>0$, the outgoing wave is
\[
u_{\rm out}(x,t)=\big(\alpha\,\delta_{x=-t},\ \beta\,\delta_{x=t},\ \beta\,\delta_{x=t}\big),
\]
where $\alpha$ is the reflection coefficient and $\beta$ the transmission coefficient. The continuity condition gives $1+\alpha=\beta$, and the Kirchhoff law gives $-(1-\alpha)+2\beta=0$. This implies
\begin{equation}\label{2.7}
\alpha=-1/3,\qquad\beta=2/3 .
\end{equation}

\begin{remark}\label{rem2.1}
For the general Bethe lattice in which each node has $N+1$ neighbours, one finds $1+\alpha=\beta$ and $-(1-\alpha)+N\beta=0$, hence
\[
\alpha=-(N-1)/(N+1),\qquad \beta=2/(N+1).
\]
One always has $\alpha+N\beta=1$ (charge conservation) and $\alpha^2+N\beta^2=1$ (energy conservation).
\end{remark}

\section{The computation of $G(k,\vB,\vA)$}\label{sec3}

Recall from \eqref{1.5} that, for $k\in\N$, $G(k,\vB,\vA)$ denotes the coefficient on the oriented edge $\vB$ of the Dirac mass propagating in the direction of $\vB$, for the solution of the wave equation whose Cauchy data is the Dirac mass at $m(A)$ propagating in the direction of $\vA$. We thus have $G(0,\vB,\vA)=\delta_{\vB=\vA}$, and from \eqref{2.7} we get
\begin{equation}\label{3.8}
G(1,\vB,\vA)=\begin{cases}-\frac13 & \text{if }\vB=-\vA,\\[1mm] \frac23 & \text{if }\vA\to\vB\text{ and }\vB\neq-\vA,\\[1mm] 0 & \text{in all other cases}.\end{cases}
\end{equation}
Observe that $\vA\to\vB$ if and only if $G(1,\vB,\vA)\neq0$. For $k\in\Z$, $\G(k)_{\vB,\vA}=G(k,\vB,\vA)$ is an infinite matrix indexed by the oriented edges, each column of which has only finitely many nonzero coefficients.

Suppose the Cauchy data at $t=0$ of a solution $f$ of the wave equation is $X(0)=\sum_{\vA}x_{\vA}\,\delta_{\vA}$, which means that for every edge $A$
\begin{equation}\label{3.9}
\begin{cases}
f_{|t=0,s\in A}=\big(x_{\vA}+x_{-\vA}\big)\,\delta_{s=\frac12},\\[1mm]
\frac{\partial f}{\partial t}_{|t=0,s\in A}=\big(-x_{\vA}+x_{-\vA}\big)\,\delta'_{s=\frac12}.
\end{cases}
\end{equation}
Then the Cauchy data at any time $t=k\in\Z$ is
\[
X(k)=\sum_{\vB}y_{\vB}\,\delta_{\vB},\qquad y_{\vB}=\sum_{\vA}G(k,\vB,\vA)\,x_{\vA}.
\]
One has the group law $\G(k+l)=\G(k)\G(l)$ for all $k,l\in\Z$, and thus
\begin{equation}\label{3.10}
\G(k)=\G(1)^k .
\end{equation}
We are thus reduced to computing the $k$-th power of the matrix $\G(1)$ given by \eqref{3.8}. We equip the space $\{X=\sum_{\vA}x_{\vA}\delta_{\vA}\}$ with the $\ell^2$ norm $\|X\|^2=\sum_{\vA}|x_{\vA}|^2$.

\begin{lemma}\label{lem3.1}
The matrix $\G(1)$ is unitary.
\end{lemma}

\begin{proof}
For any oriented edge $\vA$, \eqref{3.8} gives $\|\G(1)\delta_{\vA}\|^2=\frac19+\frac49+\frac49=1$. For $\vA\neq\vB$, let us check that $(\G(1)\delta_{\vA}|\G(1)\delta_{\vB})=0$. One has
\[
\G(1)\delta_{\vA}=-\tfrac13\delta_{-\vA}+\tfrac23\delta_{\vB_1}+\tfrac23\delta_{\vB_2},
\]
where $\vB_1$ and $\vB_2$ are the two oriented edges with $\vA\to\vB_i$, $\vB_i\neq-\vA$. Thus, for $\vA\neq\vB$, $(\G(1)\delta_{\vA}|\G(1)\delta_{\vB})\neq0$ implies $\vB=-\vB_1$ or $\vB=-\vB_2$. But
\[
\G(1)\delta_{-\vB_1}=-\tfrac13\delta_{\vB_1}+\tfrac23\delta_{-\vA}+\tfrac23\delta_{\vB_2},
\]
hence $(\G(1)\delta_{\vA}|\G(1)\delta_{-\vB_1})=-\frac29-\frac29+\frac49=0$. The same holds for $-\vB_2$. Surjectivity follows by time reversal.
\end{proof}

To compute $G(k,\vB,\vA)$ we use a signed sum over paths.

\begin{definition}\label{def3.2}
A path $\gamma$ of length $\ell(\gamma)=k$ connecting $\vA$ to $\vB$ is a $(k+1)$-tuple
\[
\gamma=(\vA,\vC_1,\dots,\vC_{k-1},\vB)=(\vC_0,\dots,\vC_k),\qquad \vA=\vC_0,\ \vB=\vC_k,
\]
with $\vC_j\to\vC_{j+1}$ for all $j\in\{0,\dots,k-1\}$. We say that $(\vC_j,\vC_{j+1})$ is an inversion if $\vC_{j+1}=-\vC_j$, and we denote by $r(\gamma)$ the number of inversions of $\gamma$. We have $0\le r(\gamma)\le k=\ell(\gamma)$.
\end{definition}

We denote by $C_{k,r}(\vA,\vB)$ the set of paths of length $k\ge1$ with $r$ inversions connecting $\vA$ to $\vB$, and we set $C_k(\vA,\vB)=\bigcup_{r\ge0}C_{k,r}(\vA,\vB)$. By definition of the product of matrices, \eqref{3.8} and \eqref{3.10} give, with $\alpha=-\frac13$ and $\beta=\frac23$,
\[
G(k,\vB,\vA)=\sum_{\gamma\in C_k(\vA,\vB)}\alpha^{r(\gamma)}\beta^{k-r(\gamma)},\qquad\forall k\ge1 .
\]
Equivalently,
\begin{equation}\label{3.11}
G(k,\vB,\vA)=\beta^k\sum_{r=0}^k\Big(\frac\alpha\beta\Big)^r\big|C_{k,r}(\vA,\vB)\big|,\qquad\forall k\ge1 .
\end{equation}
It thus remains to compute $|C_{k,r}(\vA,\vB)|$, the number of paths of length $k\ge1$ with $r$ inversions connecting $\vA$ to $\vB$.

\subsection*{Computation of $|C_{k,r}(\vA,\vB)|$}
By symmetry, $|C_{k,r}(\vA,\vB)|$ depends only on the distance $d=\dist(m(A),m(B))$ and on the orientations of $\vA$ and $\vB$ with respect to the unique geodesic connecting $m(A)$ to $m(B)$. There are four possible configurations, shown in Figure~\ref{fig4}.

\begin{figure}[ht]
\centering
\begin{tikzpicture}[scale=0.95,>=Stealth]
\foreach \y/\lab/\da/\db in {0/{Case 1: $\vA$ incoming, $\vB$ outgoing}/1/1, -0.8/{Case 2: $\vA$ incoming, $\vB$ incoming}/1/-1, -1.6/{Case 3: $\vA$ outgoing, $\vB$ incoming}/-1/-1, -2.4/{Case 4: $\vA$ outgoing, $\vB$ outgoing}/-1/1}{
 \draw (0,\y)--(6,\y);
 \foreach \x in {0,1,5,6} \fill (\x,\y) circle (1.2pt);
 \draw[->,thick] ({0.5-0.25*\da},\y+0.12)--({0.5+0.25*\da},\y+0.12);
 \draw[->,thick] ({5.5-0.25*\db},\y+0.12)--({5.5+0.25*\db},\y+0.12);
 \node[below] at (0.5,\y) {\scriptsize $\vA$}; \node[below] at (5.5,\y) {\scriptsize $\vB$};
 \node[right] at (6.3,\y) {\small \lab};
}
\draw[decorate,decoration={brace,amplitude=4pt}] (1,0.35)--(5,0.35) node[midway,above=4pt] {\scriptsize geodesic from $m(A)$ to $m(B)$};
\end{tikzpicture}
\caption{The four possible orientations of $(\vA,\vB)$.}
\label{fig4}
\end{figure}

We denote by $\Gamma_{k,r}(d,j)$, $j\in\{1,2,3,4\}$, the number of paths of length $k$ with $r$ inversions connecting $\vA$ to $\vB$, such that $\dist(m(A),m(B))=d$ and the orientation of $(\vA,\vB)$ is of type $j$ as in Figure~\ref{fig4}. In the particular case $d=0$, one has $\Gamma_{k,r}(0,1)=\Gamma_{k,r}(0,3)$ and $\Gamma_{k,r}(0,2)=\Gamma_{k,r}(0,4)$; type $1$ (or $3$) corresponds to $\vB=\vA$, and type $2$ (or $4$) to $\vB=-\vA$. We compute $\Gamma_{k,r}(d,j)$ by induction on $k$.

\emph{Initialization: $k=1$.} Then either $(d,r)=(1,0)$ or $(d,r)=(0,1)$, and we find
\begin{equation}\label{3.12}
\begin{cases}
\Gamma_{1,0}(1,1)=1,\\
\Gamma_{1,1}(0,2)=1=\Gamma_{1,1}(0,4),\\
\Gamma_{1,r}(d,j)=0\ \text{in all other cases.}
\end{cases}
\end{equation}
\emph{Induction: $k-1\to k$, with $k\ge2$} and the convention $\Gamma_{k,r}(d,j)=0$ if $r<0$, $d>k$ or $r>k$.

\begin{lemma}\label{lem3.3}
For $k\ge2$, the following formulas hold.
\begin{enumerate}[label=(\arabic*),leftmargin=*]
\item For $d\ge1$,
\begin{equation}\label{3.13}
\begin{cases}
\Gamma_{k,r}(d,1)=\Gamma_{k-1,r-1}(d,2)+\Gamma_{k-1,r}(d-1,1)+\Gamma_{k-1,r}(d,2),\\
\Gamma_{k,r}(d,2)=\Gamma_{k-1,r-1}(d,1)+2\Gamma_{k-1,r}(d+1,2),\\
\Gamma_{k,r}(d,3)=\Gamma_{k-1,r-1}(d,4)+2\Gamma_{k-1,r}(d+1,3),\\
\Gamma_{k,r}(d,4)=\Gamma_{k-1,r-1}(d,3)+\Gamma_{k-1,r}(d-1,4)+\Gamma_{k-1,r}(d,3).
\end{cases}
\end{equation}
\item For $d=0$,
\begin{equation}\label{3.14}
\begin{cases}
\Gamma_{k,r}(0,1)=\Gamma_{k-1,r-1}(0,2)+2\Gamma_{k-1,r}(1,3),\\
\Gamma_{k,r}(0,2)=\Gamma_{k-1,r-1}(0,1)+2\Gamma_{k-1,r}(1,2),\\
\Gamma_{k,r}(0,3)=\Gamma_{k-1,r-1}(0,4)+2\Gamma_{k-1,r}(1,3),\\
\Gamma_{k,r}(0,4)=\Gamma_{k-1,r-1}(0,3)+2\Gamma_{k-1,r}(1,2).
\end{cases}
\end{equation}
\end{enumerate}
\end{lemma}

\begin{proof}
Let us check the first line of \eqref{3.13}. Let $\gamma=(\vA,\vC_1,\dots,\vC_{k-1},\vB)$ be a path of length $k$ with $r$ inversions, with $(\vA,\vB)$ of type $1$. The three oriented edges $\vC$ with $\vC\to\vB$ are $-\vB$, $\vB_1$ and $\vB_2$, where $\vB_1$ lies on the geodesic from $m(A)$ and $\vB_2$ does not. There are three possibilities.
\begin{enumerate}[leftmargin=*]
\item If $\vC_{k-1}=-\vB$, then $(\vA,\vC_1,\dots,\vC_{k-1})$ is a path of length $k-1$ with $r-1$ inversions. One has $\dist(m(A),m(C_{k-1}))=d$, and the orientation of $(\vA,\vC_{k-1})$ is of type $2$.
\item If $\vC_{k-1}=\vB_1$, with $\dist(m(A),m(B_1))=d-1$, then $(\vA,\vC_1,\dots,\vC_{k-1})$ is a path of length $k-1$ with $r$ inversions, of type $1$.
\item If $\vC_{k-1}=\vB_2$, with $\dist(m(A),m(B_2))=d$, then $(\vA,\vC_1,\dots,\vC_{k-1})$ is a path of length $k-1$ with $r$ inversions, of type $2$.
\end{enumerate}
This proves the first line of \eqref{3.13}. The verification of the other seven formulas of \eqref{3.13} and \eqref{3.14} is the same, and we leave it to the reader.
\end{proof}

We define $\Gamma_{0,r}(d,j)$ by
\begin{equation}\label{3.15}
\Gamma_{0,0}(0,1)=\Gamma_{0,0}(0,3)=1,\qquad\Gamma_{0,r}(d,j)=0\ \text{in all other cases}.
\end{equation}
This convention for $k=0$ is compatible with \eqref{3.12}, \eqref{3.13} and \eqref{3.14}. We then define the generating functions $\F_j$, $j\in\{1,2,3,4\}$, as the formal series in $X,Y,Z$
\begin{equation}\label{3.16}
\F_j=\sum_{k\ge0,\,d\ge0,\,r\ge0}\Gamma_{k,r}(d,j)\,X^kY^rZ^d .
\end{equation}
From \eqref{3.12}, \eqref{3.15}, \eqref{3.13} and \eqref{3.14}, we deduce the system
\begin{equation}\label{3.17}
\begin{cases}
\F_1=1+XZ\F_1+XY\F_2+X\big[\F_2-\F_{2|Z=0}\big]+2X\partial_Z\F_{3|Z=0},\\[1mm]
\F_2=XY\F_1+\frac{2X}{Z}\big[\F_2-\F_{2|Z=0}\big],\\[1mm]
\F_3=1+XY\F_4+\frac{2X}{Z}\big[\F_3-\F_{3|Z=0}\big],\\[1mm]
\F_4=XZ\F_4+XY\F_3+X\big[\F_3-\F_{3|Z=0}\big]+2X\partial_Z\F_{2|Z=0}.
\end{cases}
\end{equation}
The following lemma is obvious.

\begin{lemma}\label{lem3.4}
Let $\Theta$ be the operator acting on formal series in $(X,Y,Z)$ by
\[
\Theta(F)=F-\frac{2X}{Z}\big(F-F_{|Z=0}\big).
\]
Then $\Theta$ commutes with multiplication by $X$ and $Y$, $\Theta(1)=1$, and the equation
\[
\Theta\Big(\sum_{n\ge0}F_n(X,Y)Z^n\Big)=\sum_{n\ge0}G_n(X,Y)Z^n
\]
has the unique solution
\begin{equation}\label{3.18}
F=\sum_{n\ge0}F_n(X,Y)Z^n,\qquad F_n=\sum_{p\ge0}(2X)^pG_{n+p}.
\end{equation}
\end{lemma}

From the second and third lines of \eqref{3.17},
\begin{equation}\label{3.19}
\begin{aligned}
\F_2&=\Theta^{-1}(XY\F_1)=XY\,\Theta^{-1}(\F_1),\\
\F_3&=\Theta^{-1}(1+XY\F_4)=1+XY\,\Theta^{-1}(\F_4).
\end{aligned}
\end{equation}
Let us define the operators $\mathbb P$ and $\mathbb Q$ by
\begin{align*}
\mathbb P(F)&=F-XZF-X^2\big[Y^2+Y\big]\Theta^{-1}(F)+X^2Y\,\Theta^{-1}(F)_{|Z=0},\\
\mathbb Q(F)&=2X^2Y\,\frac{\partial}{\partial Z}\Theta^{-1}(F)_{|Z=0}.
\end{align*}
From \eqref{3.17} we get the following lemma.

\begin{lemma}\label{lem3.5}
$\F_1$ and $\F_4$ satisfy the system
\begin{equation}\label{3.20}
\begin{bmatrix}\mathbb P&-\mathbb Q\\-\mathbb Q&\mathbb P\end{bmatrix}\begin{bmatrix}\F_1\\\F_4\end{bmatrix}=\begin{bmatrix}1\\XY\end{bmatrix}.
\end{equation}
\end{lemma}

By \eqref{3.11}, when the orientation of $(\vA,\vB)$ is of type $j\in\{1,2,3,4\}$ and $\dist(m(A),m(B))=d$,
\begin{equation}\label{3.21}
G(k,\vB,\vA)=\Big(\frac23\Big)^k\frac1{k!}\partial_X^k\frac1{d!}\partial_Z^d\F_{j\,|X=Z=0,\,Y=-\frac12}.
\end{equation}
Observe that $Y$ is just a parameter in the system \eqref{3.20}. Since we are in the case $Y=\alpha/\beta=-\frac12$, we may eliminate the variable $Y$ by setting $Y=-\frac12$. Then $\F_1$ and $\F_4$ are formal series in $X,Z$, and they satisfy
\begin{equation}\label{3.22}
\begin{cases}
\begin{bmatrix}\mathbb P&-\mathbb Q\\-\mathbb Q&\mathbb P\end{bmatrix}\begin{bmatrix}\F_1\\\F_4\end{bmatrix}=\begin{bmatrix}1\\-\frac X2\end{bmatrix},\\[4mm]
\mathbb P=1-XZ+\frac{X^2}4\Theta^{-1}-\frac{X^2}2\Theta^{-1}(\cdot)_{|Z=0},\qquad
\mathbb Q=-X^2\partial_Z\Theta^{-1}(\cdot)_{|Z=0}.
\end{cases}
\end{equation}
We are thus reduced to solving the system \eqref{3.22}. By \eqref{3.19}, $\F_2$ and $\F_3$ are then given by
\begin{equation}\label{3.23}
\F_2=-\frac X2\Theta^{-1}(\F_1),\qquad \F_3=1-\frac X2\Theta^{-1}(\F_4).
\end{equation}

\subsection*{Resolution of the system \eqref{3.22}}
With the change of variables $X=\sqrt2x$, $Z=\sqrt2z$, \eqref{3.22} becomes
\begin{equation}\label{3.24}
\begin{cases}
\begin{bmatrix}\tilde{\mathbb P}&-\tilde{\mathbb Q}\\-\tilde{\mathbb Q}&\tilde{\mathbb P}\end{bmatrix}\begin{bmatrix}\tilde\F_1\\\tilde\F_4\end{bmatrix}=\begin{bmatrix}1\\-\frac{x}{\sqrt2}\end{bmatrix},\\[4mm]
\tilde{\mathbb P}=1-2xz+\frac{x^2}2\Theta^{-1}-x^2\Theta^{-1}(\cdot)_{|z=0},\qquad
\tilde{\mathbb Q}=-\sqrt2x^2\partial_z\Theta^{-1}(\cdot)_{|z=0},
\end{cases}
\end{equation}
with $\Theta(f)=f-\frac{2x}{z}(f-f_{|z=0})$. With the new unknowns $\tilde\F_1=\Theta(f_1)$ and $\tilde\F_4=\Theta(f_4)$, and with
\[
r=\frac{\sqrt2}3,\qquad \mu=\mu(x)=\frac{2x}{1+\frac{x^2}{r^2}},\qquad \delta=\delta(x)=\frac{x^2}{1+\frac{x^2}{r^2}}=\frac x2\mu,
\]
the system \eqref{3.24} reads
\begin{equation}\label{3.25}
\begin{cases}
\begin{bmatrix}p&-q\\-q&p\end{bmatrix}\begin{bmatrix}f_1\\f_4\end{bmatrix}=\dfrac{1}{1+\frac{9x^2}2}\begin{bmatrix}1\\-\frac{x}{\sqrt2}\end{bmatrix},\\[4mm]
p(f)=\Big[1-\mu z-\frac\mu z\Big]f+\Big(\frac\mu z-5\delta\Big)f_{|z=0},\qquad q(f)=-\sqrt2\,\delta\,\partial_zf_{|z=0}.
\end{cases}
\end{equation}
The system \eqref{3.25} is an equation on functions of $z$, with $x$ as a holomorphic parameter in the complex disc $\{x\in\C,\ |x|<r=\sqrt2/3\}$. For $x=0$ one has $p={\rm Id}$ and $q=0$.

\begin{definition}\label{def3.6}
Let $U$ be the open subset of $\C$
\[
U=\C\setminus\big\{\mu\in\C:\ \mu^2\in[\tfrac14,+\infty[\big\}.
\]
For $\mu\in U$ we define the holomorphic function
\[
z_-(\mu)=\frac{1}{2\mu}\Big[1-\sqrt{1-4\mu^2}\Big],
\]
where $\sqrt{1-4\mu^2}$ is the principal determination of the square root of $1-4\mu^2$ on $U$. One has $z_-(0)=0$, $z_-'(0)=1$ and $z_-^2-\frac{z_-}{\mu}+1=0$. We also define $z_+(\mu)=\frac1{z_-(\mu)}=\frac1{2\mu}\big[1+\sqrt{1-4\mu^2}\big]$; it is meromorphic on $U$, with a simple pole at $\mu=0$.

Let $V$ be the open subset of $\C$
\[
V=\C\setminus\big\{\mu\in\C:\ \mu^2\in[r^2,+\infty[\big\}.
\]
Since $r=\frac{\sqrt2}3<\frac12$, one has $V\subset U$. We denote by $\theta_0$ the unique solution of $\cos\theta_0=2r$ with $\theta_0\in\,]0,\frac\pi2[$.
\end{definition}

\begin{lemma}\label{lem3.7}
The holomorphic function $x\mapsto\mu(x)=\frac{2x}{1+x^2/r^2}$ is a bijection from $D_r=\{x:\ |x|<r\}$ onto $V$. Its restriction to the boundary $\partial D_r\setminus\{\pm ir\}$ satisfies
\[
\mu\big(\{re^{i\theta}:\ -\tfrac\pi2<\theta<\tfrac\pi2\}\big)=[r,+\infty[,\qquad
\mu\big(\{re^{i\theta}:\ \tfrac\pi2<\theta<\tfrac{3\pi}2\}\big)=\,]-\infty,-r].
\]
\end{lemma}

\begin{proof}
For $\mu\neq0$, the equation $\mu(x)=\mu$ is equivalent to $x^2-2x\frac{r^2}{\mu}+r^2=0$. It has two solutions $x_\pm(\mu)$ with $x_+x_-=r^2$. Thus, when $|x_+|\neq|x_-|$, there is a unique solution in $D_r$. If $|x_+|=|x_-|=r$, then $x_+=re^{i\varphi}$, $x_-=re^{-i\varphi}$ and $x_++x_-=2r\cos\varphi=\frac{2r^2}\mu$, so $\mu\in\R$ and $|\mu|=|r/\cos\varphi|\ge r$. Since $\mu(x)=0$ if and only if $x=0$, $x\mapsto\mu(x)$ is a holomorphic bijection from $D_r$ onto $V$. The last two equalities follow from
\[
\mu(re^{i\theta})=\frac{2re^{i\theta}}{1+e^{2i\theta}}=\frac{r}{\cos\theta}. \qedhere
\]
\end{proof}

\begin{figure}[ht]
\centering
\begin{tikzpicture}[scale=1.25,>=Stealth]
\begin{scope}
\draw[->] (-1.6,0)--(1.6,0); \draw[->] (0,-1.5)--(0,1.5);
\draw (0,0) circle (1);
\foreach \a in {19.47,160.53,199.47,340.53}{\fill (\a:1) circle (1.2pt); \draw[dashed,gray] (0,0)--(\a:1);}
\node[above right] at (19.47:1) {\scriptsize $\mu=\frac12$};
\node[above left] at (160.53:1) {\scriptsize $\mu=-\frac12$};
\node[below left] at (199.47:1) {\scriptsize $\mu=-\frac12$};
\node[below right] at (340.53:1) {\scriptsize $\mu=\frac12$};
\fill (1,0) circle (1.2pt); \node[below right] at (1,0) {\scriptsize $\mu=r$};
\fill (-1,0) circle (1.2pt); \node[below left] at (-1,0) {\scriptsize $\mu=-r$};
\node[right] at (0.05,1.12) {\scriptsize $\mu=+\infty$};
\node[right] at (0.05,-1.12) {\scriptsize $\mu=-\infty$};
\draw (0.4,0) arc (0:19.47:0.4); \node at (0.58,0.1) {\scriptsize $\theta_0$};
\node at (0,1.75) {\small $x\in D_r$};
\end{scope}
\draw[->] (1.8,-1.3) to[bend right=25] node[below]{\small $\mu$} (3.6,-1.3);
\begin{scope}[xshift=5.4cm]
\draw[->] (-2,0)--(2,0); \draw[->] (0,-1.5)--(0,1.5);
\draw[line width=1.6pt] (0.8,0)--(2,0); \draw[line width=1.6pt] (-0.8,0)--(-2,0);
\foreach \x/\l in {0.8/r,1.2/\frac12,-0.8/-r,-1.2/-\frac12}{\fill (\x,0) circle (1.2pt); \node[below] at (\x,-0.05) {\scriptsize $\l$};}
\node at (0,1.75) {\small $\mu\in V$};
\end{scope}
\end{tikzpicture}
\caption{The map $x\mapsto\mu(x)$ from $D_r$ onto $V$. The cuts $\mu^2\ge r^2$ are drawn in bold.}
\label{fig5}
\end{figure}

\begin{lemma}\label{lem3.8}
For all $\mu\in U$ one has $|z_-(\mu)|<1$.
\end{lemma}

\begin{proof}
One has $z_-(0)=0$, and for $\mu\neq0$ both $z_-(\mu)$ and $z_+(\mu)=1/z_-(\mu)$ satisfy $z^2-\frac z\mu+1=0$. Thus $|z_-(\mu)|<1$ for $\mu$ near $0$. Suppose there exists $\mu\in U$ such that $|z_-(\mu)|=1$. Then $|z_+(\mu)|=1$, so $z_-(\mu)=e^{i\varphi}$ and $z_+(\mu)=e^{-i\varphi}$. The relation $z_++z_-=\frac1\mu$ then gives $\cos\varphi=\frac1{2\mu}$, hence $\mu\in\R$ and $|2\mu|\ge1$, which contradicts $\mu\in U$.
\end{proof}

\begin{lemma}\label{lem3.9}
Let $x\in\C$ be close to $0$, and let $p,q$ be the operators defined in \eqref{3.25}. For $(\alpha,\beta)\in\C^2$, the system
\begin{equation}\label{3.26}
\begin{cases}p(f)-q(g)=\alpha,\\-q(f)+p(g)=\beta\end{cases}
\end{equation}
has a unique solution $(f,g)=\big(\sum_{n\ge0}f_nz^n,\sum_{n\ge0}g_nz^n\big)$ holomorphic in $|z|<1$, given by
\begin{equation}\label{3.27}
f=\frac{a\,z_-(\mu)}{\mu}\,\frac1{1-z\,z_-(\mu)},\qquad g=\frac{b\,z_-(\mu)}{\mu}\,\frac1{1-z\,z_-(\mu)},
\end{equation}
where $(a,b)\in\C^2$ is the solution of
\begin{equation}\label{3.28}
\begin{bmatrix}1-5\delta\frac{z_-}{\mu}&\sqrt2\,\delta\frac{z_-^2}{\mu}\\[1mm]\sqrt2\,\delta\frac{z_-^2}{\mu}&1-5\delta\frac{z_-}{\mu}\end{bmatrix}\begin{bmatrix}a\\b\end{bmatrix}=\begin{bmatrix}\alpha\\\beta\end{bmatrix}.
\end{equation}
\end{lemma}

\begin{proof}
When $x=0$, one has $\mu=0$, $z_-(0)=0$, $\frac{z_-(\mu)}{\mu}=1$, $\delta=0$, $p={\rm Id}$ and $q=0$, and the result is obvious. We may thus assume $x\neq0$, hence $\mu\neq0$. Multiplying \eqref{3.26} by $z$, we get
\begin{align*}
\big(z-\mu(z^2+1)\big)f&=az-c, & a&=\alpha+5\delta f_0-\sqrt2\delta g_1, & c&=\mu f_0,\\
\big(z-\mu(z^2+1)\big)g&=bz-d, & b&=\beta+5\delta g_0-\sqrt2\delta f_1, & d&=\mu g_0.
\end{align*}
Since $z-\mu(z^2+1)=-\mu(z-z_-)(z-z_+)$ and $f,g$ are holomorphic in $|z|<1$, we must have $c=az_-$ and $d=bz_-$. Thus
\[
f=\frac{-a}{\mu(z-z_+)}=\frac{a}{\mu(z_+-z)}=\frac{a}{\mu z_+(1-\frac z{z_+})}=\frac{az_-}\mu\,\frac1{1-zz_-},
\qquad g=\frac{bz_-}\mu\,\frac1{1-zz_-}.
\]
Therefore \eqref{3.27} holds. Evaluating at $z=0$ gives
\[
f_0=\frac{az_-}\mu,\quad g_0=\frac{bz_-}\mu,\quad f_1=\frac{az_-^2}\mu,\quad g_1=\frac{bz_-^2}\mu .
\]
Hence the unknowns $(a,b)$ satisfy
\[
a=\alpha+5\delta\frac{az_-}\mu-\sqrt2\delta\frac{bz_-^2}\mu,\qquad b=\beta+5\delta\frac{bz_-}\mu-\sqrt2\delta\frac{az_-^2}\mu,
\]
which is equivalent to \eqref{3.28}. Finally, since $|x|$ is small, the matrix in \eqref{3.28} is invertible.
\end{proof}

Recall that $\delta=\frac x2\mu(x)$ and $\mu=\mu(x)=\frac{2x}{1+x^2/r^2}$. We denote by $(A(x),B(x))=(a,b)$ the solution of \eqref{3.28} whose right-hand side is associated with \eqref{3.25}, that is,
\[
(\alpha,\beta)=\Big(\frac{1}{1+\frac{x^2}{r^2}},\ \frac{-x}{\sqrt2\,(1+\frac{x^2}{r^2})}\Big).
\]
Then $A(x)$ and $B(x)$ are well-defined holomorphic functions near $x=0$. Solving \eqref{3.28} gives, with $z_-=z_-(\mu(x))$,
\begin{equation}\label{3.29}
\begin{aligned}
A(x)&=\frac{1}{1+\frac{x^2}{r^2}}\cdot\frac{1}{D(x)}\Big[1-\frac52xz_-+\frac{x^2}2z_-^2\Big],\\
B(x)&=-\frac{x}{\sqrt2}\cdot\frac{1}{1+\frac{x^2}{r^2}}\cdot\frac{1}{D(x)}\Big[1-\frac52xz_-+z_-^2\Big],
\end{aligned}
\qquad
D(x)=\Big(1-\frac52xz_-\Big)^2-\frac{x^2}{2}z_-^4 .
\end{equation}
We define the functions $g_j(x)$, $j\in\{1,2,3,4\}$, holomorphic near $x=0$, by
\begin{align*}
g_1(x)&=A(x)\,\frac{z_-(\mu(x))\,(1-2xz_-(\mu(x)))}{\mu(x)}, & g_4(x)&=B(x)\,\frac{z_-(\mu(x))\,(1-2xz_-(\mu(x)))}{\mu(x)},\\
g_2(x)&=-\frac{xA(x)}{\sqrt2}\,\frac{z_-(\mu(x))}{\mu(x)}, & g_3(x)&=-\frac{xB(x)}{\sqrt2}\,\frac{z_-(\mu(x))}{\mu(x)}.
\end{align*}
From $\tilde\F_1=\Theta(f_1)$, $\tilde\F_4=\Theta(f_4)$ and \eqref{3.23}, we get
\begin{equation}\label{3.30}
\begin{aligned}
\tilde\F_1(x,z)&=\frac{g_1(x)}{1-zz_-(\mu(x))}, & \tilde\F_2(x,z)&=\frac{g_2(x)}{1-zz_-(\mu(x))},\\
\tilde\F_3(x,z)&=1+\frac{g_3(x)}{1-zz_-(\mu(x))}, & \tilde\F_4(x,z)&=\frac{g_4(x)}{1-zz_-(\mu(x))}.
\end{aligned}
\end{equation}
Since $\mu(x)=2x(1+x^2/r^2)^{-1}$, we find the following formulas for the functions $g_j$, with $z_-=z_-(\mu(x))$ and $D$ as in \eqref{3.29}:
\begin{equation}\label{3.31}
\begin{aligned}
g_1(x)&=\frac{z_-}{2x}\,\frac{(1-2xz_-)\big(1-\frac52xz_-+\frac{x^2}2z_-^2\big)}{D(x)}, &
g_2(x)&=\frac{-z_-}{2\sqrt2}\,\frac{1-\frac52xz_-+\frac{x^2}2z_-^2}{D(x)},\\
g_3(x)&=\frac{xz_-}{4}\,\frac{1-\frac52xz_-+z_-^2}{D(x)}, &
g_4(x)&=\frac{-z_-}{2\sqrt2}\,\frac{(1-2xz_-)\big(1-\frac52xz_-+z_-^2\big)}{D(x)}.
\end{aligned}
\end{equation}
When the orientation of $(\vA,\vB)$ is of type $j\in\{1,2,3,4\}$, we write $G(k,\vB,\vA)=G_j(k,d)$, where $d=\dist(m(A),m(B))$.

\begin{proposition}\label{prop3.10}
For $k\ge1$ and $j\in\{1,2,3,4\}$, one has
\begin{equation}\label{3.32}
G_j(k,d)=\Big(\frac{\sqrt2}3\Big)^k\Big(\frac1{\sqrt2}\Big)^d\frac1{k!}\partial_x^k\Big[g_j(x)\,z_-^d(\mu(x))\Big]_{|x=0}.
\end{equation}
\end{proposition}

\begin{proof}
This follows from \eqref{3.21}, \eqref{3.30} and the change of variables $Z=\sqrt2z$, $X=\sqrt2x$.
\end{proof}

We also denote by
\[
\EE_j(k,d)=\sum_{(\vA,\vB)\ \text{of type }j,\ \vA\ \text{fixed},\ \dist(m(A),m(B))=d}|G(k,\vB,\vA)|^2
\]
the contribution to the energy at time $k$ of the oriented edges $\vB$ at distance $d$ of $m(A)$ such that $(\vA,\vB)$ is of type $j$.

For $d\ge1$, one has $\EE(k,d)=\sum_{j=1}^4\EE_j(k,d)$, and the number of oriented edges $\vB$ with $\dist(m(A),m(B))=d$ and $(\vA,\vB)$ of a given type $j$ is $2^d$. For $d=0$, types $1$ and $3$ both describe $\vB=\vA$, and types $2$ and $4$ both describe $\vB=-\vA$, so that $\EE(k,0)=\EE_1(k,0)+\EE_2(k,0)$. In all cases, \eqref{3.32} gives
\begin{equation}\label{3.33}
\EE_j(k,d)=2^d|G_j(k,d)|^2=\Big|\frac{r^k}{k!}\partial_x^k\big(g_j(x)z_-^d(\mu(x))\big)_{|x=0}\Big|^2,\qquad r=\frac{\sqrt2}3 .
\end{equation}

\section{Proof of Theorem~\ref{thm1.1}, items (1) and (2)}\label{sec4}

In this section we study the functions $d\in\{0,1,\dots,k\}\mapsto\EE_j(k,d)$. We start with the analytic structure of the functions $g_j$.

\begin{lemma}\label{lem4.1}
The functions $g_j(x)$ are holomorphic in the complex disc $D_r=\{x:\ |x|<r\}$. They extend analytically across $\partial D_r$, except at the six points
\[
x_0\in\{re^{i\theta_0},\,re^{-i\theta_0},\,-re^{i\theta_0},\,-re^{-i\theta_0}\}\qquad\text{and}\qquad x=\pm r,
\]
where $\cos\theta_0=2r$.
\begin{enumerate}[label=(\roman*),leftmargin=*]
\item Near each $x_0$, $g_j(x)=a_j(x)+(x-x_0)^{1/2}b_j(x)$, with $a_j,b_j$ holomorphic near $x_0$.
\item At $x=\pm r$, $g_j$ has a simple pole:
\[
g_j(x)=\frac{A_j^+}{1-x/r}+\frac{A_j^-}{1+x/r}+\big(\text{holomorphic near }\pm r\big),
\]
with
\[
A_1^\pm=A_3^\pm=\tfrac16,\qquad A_2^+=A_4^+=-\tfrac16,\qquad A_2^-=A_4^-=\tfrac16 .
\]
\end{enumerate}
\end{lemma}

\begin{proof}
\emph{Holomorphy in $D_r$.} Since we know a priori that $\sum_{d=0}^k\EE_j(k,d)\le1$, \eqref{3.33} with $d=0$ implies that the functions $x\mapsto g_j(x)$ are holomorphic in the complex disc $D_r$. More precisely, writing $g_j(x)=\sum_{n\ge0}g_{j,n}x^n$, \eqref{3.33} gives
\begin{equation}\label{4.34}
|g_{j,n}|=\Big|\frac1{n!}\partial_x^ng_j(0)\Big|\le r^{-n}\sqrt{\EE_j(n,0)}\le r^{-n}.
\end{equation}
This a priori bound does not exclude poles on $\partial D_r$; a simple pole at $x=r$ has coefficients of size exactly $r^{-n}$.

\emph{The function $h$.} To analyse the singularities of $g_j$ on $\partial D_r$, we use \eqref{3.31}. By Lemmas~\ref{lem3.7} and \ref{lem3.8}, the function $h(x)=z_-(\mu(x))$ is holomorphic in $D_r$ and satisfies $|h(x)|<1$. Moreover $h(0)=0$, $h'(0)=1$ and
\[
h^2-\Big(1+\frac{x^2}{r^2}\Big)\frac{h}{2x}+1=0 .
\]
This shows that $h(x)=0$ if and only if $x=0$. The only possible singularities of $h$ are the zeros of the discriminant of this quadratic equation, that is $1+\frac{x^2}{r^2}=\pm4x$, which is equivalent to $x\in\{re^{\pm i\theta_0},-re^{\pm i\theta_0}\}$. Near these points, $h(x)=a(x)+(x-x_0)^{1/2}b(x)$ with $a,b$ holomorphic near $x_0$.

\emph{Zeros of $D$.} It remains to locate the zeros of the common denominator $D(x)=(1-\frac52xz_-)^2-\frac{x^2}2z_-^4$ in $\overline{D_r}$. We shall prove
\begin{equation}\label{4.35}
|x|\le r,\quad x\neq\pm r\ \Longrightarrow\ D(x)\neq0,\qquad\text{and}\qquad D(\pm r)=0 .
\end{equation} If $D(x)=0$, then $1-\frac52xz_-=\varepsilon\frac{x}{\sqrt2}z_-^2$ with $\varepsilon=\pm1$, so
\begin{equation}\label{4.36}
\begin{cases}
\frac{\varepsilon x}{\sqrt2}z_-^2+\frac52xz_--1=0,\\[1mm]
z_-^2-\frac{z_-}\mu+1=0.
\end{cases}
\end{equation}
We eliminate $z_-$ in the polynomial system \eqref{4.36}. With $x=ry=\frac{\sqrt2}3y$, the resultant in $z_-$ of the two equations (the second one being multiplied by $\mu(1+y^2)$) is
\begin{equation}\label{4.37}
\frac1{27}\,y\,(y-1)(y+1)(y+3\varepsilon)^2 .
\end{equation}
The root $y=0$ is spurious, since the first equation of \eqref{4.36} reads $-1=0$ there, and the roots $y=\mp3\varepsilon$ give $|x|=3r>r$. Hence the only zeros of $D$ in $\overline{D_r}$ are at $y=\pm1$, that is $x=\pm r$.

At these points, $\mu(\pm r)=\pm r$, hence $\sqrt{1-4\mu^2}=\frac13$ and
\[
z_-(\mu(\pm r))=\pm\frac1{\sqrt2},\qquad xz_-=\frac13,\qquad 1-\frac52xz_-=\frac16,\qquad \frac{x^2}2z_-^4=\frac1{36},
\]
so that $D(\pm r)=0$. Indeed $(x,z_-,\varepsilon)=(r,\frac1{\sqrt2},+1)$ and $(-r,-\frac1{\sqrt2},-1)$ solve \eqref{4.36}. These zeros are simple, and the numerators of \eqref{3.31} do not vanish there; see Proposition~\ref{prop5.1}, which also gives the values of $A_j^\pm$. This proves \eqref{4.35}. Since $z_-(\mu(x))$ is holomorphic near $\pm r$ (because $\mu(\pm r)=\pm r\in U$), $g_j$ has a simple pole at $\pm r$, and is holomorphic near every other point of $\partial D_r\setminus\{x_0\}$.
\end{proof}

\begin{remark}\label{rem4.2}
The poles at $x=\pm r$ correspond to $\theta=0$ and $\theta=\pi$ in the parametrization $x=re^{i\theta}$. These are the centres of the two arcs $|\cos\theta|>2r$, on which $z_-(\mu(re^{i\theta}))$ is real, with $|z_-|<1$. Their meaning is explained in Section~\ref{sec5}: they correspond to the eigenvalues $+1$ and $-1$ of the unitary operator $\G(1)$.
\end{remark}

We define $\psi_j(k,d)$, $j\in\{1,2,3,4\}$, such that $\EE_j(k,d)=|\psi_j(k,d)|^2$, by
\begin{equation}\label{4.38}
\psi_j(k,d)=r^k\,\frac1{k!}\,\partial_x^k\big[g_j(x)z_-^d(\mu(x))\big]_{|x=0}.
\end{equation}
By the proof of Lemma~\ref{lem4.1}, the function $h(x)=z_-(\mu(x))$ is holomorphic in $D_r$ and continuous on $\overline{D_r}$. It satisfies $|h(x)|\le1$ and $h(x)=0$ if and only if $x=0$, and its singularities on $\partial D_r$ are at $x\in\{re^{\pm i\theta_0},-re^{\pm i\theta_0}\}$. Therefore there exists a function $\varphi(\theta)$ with the following properties:
\begin{itemize}
\item $\varphi$ is holomorphic in $\{\operatorname{Im}\theta>0\}$ and continuous on $\{\operatorname{Im}\theta\ge0\}$, with $\varphi(\theta+2\pi)=\varphi(\theta)+2\pi$ and $\operatorname{Im}\varphi(\theta)\ge0$;
\item its singularities are at $\theta\in\{\pm\theta_0,\pm(\theta_0+\pi)\}$;
\item $z_-(\mu(re^{i\theta}))=z_-\big(\frac r{\cos\theta}\big)=e^{i\varphi(\theta)}$.
\end{itemize}
We normalize it by $\varphi(\theta_0)=0$. From $z_-=e^{i\varphi}$, $z_+=e^{-i\varphi}$ and $z_++z_-=\frac1\mu=\frac{\cos\theta}r$, we get
\[
2r\cos\varphi(\theta)=\cos\theta .
\]
This gives the following explicit formulas for $\varphi(\theta)$:
\begin{equation}\label{4.39}
\begin{aligned}
\varphi(\theta)&=i\,\argch\Big(\frac{\cos\theta}{2r}\Big) && \text{for }\theta\in[-\theta_0,\theta_0],\\
\varphi(\theta)&=\arccos\Big(\frac{\cos\theta}{2r}\Big) && \text{for }\theta\in[\theta_0,\pi-\theta_0],\\
\varphi(\theta)&=\pi+i\,\argch\Big(\frac{|\cos\theta|}{2r}\Big) && \text{for }\theta\in[\pi-\theta_0,\pi+\theta_0],\\
\varphi(\theta)&=\pi+\arccos\Big(\frac{-\cos\theta}{2r}\Big) && \text{for }\theta\in[\pi+\theta_0,2\pi-\theta_0].
\end{aligned}
\end{equation}
In particular $\varphi(\theta+\pi)=\varphi(\theta)+\pi$. The variations of $\theta\mapsto\varphi(\theta)$ are shown in Figure~\ref{fig6}.

\begin{figure}[ht]
\centering
\includegraphics[width=0.9\textwidth]{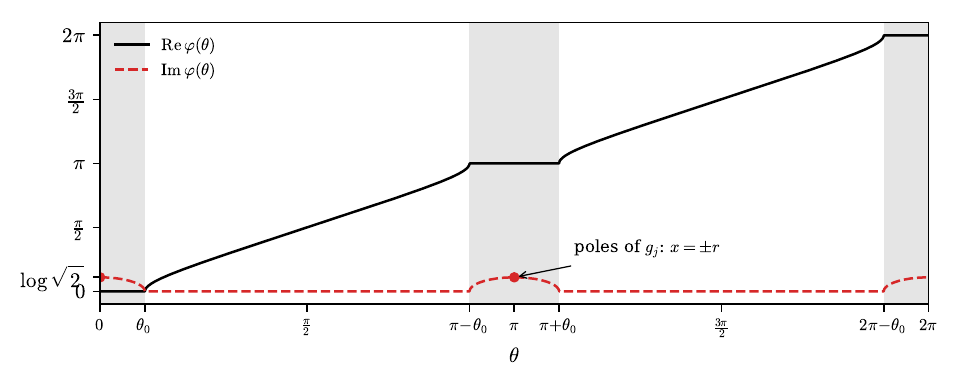}
\caption{Real and imaginary parts of $\varphi(\theta)$. On the shaded arcs $|\cos\theta|>2r$, $\varphi$ has a constant real part and a positive imaginary part, so that $|z_-|=e^{-\operatorname{Im}\varphi}<1$. At the centres $\theta=0,\pi$ of these arcs, $\operatorname{Im}\varphi=\log\sqrt2$ and $z_-=\pm2^{-1/2}$; there the generating functions $g_j$ have the poles $x=\pm r$.}
\label{fig6}
\end{figure}

\begin{figure}[ht]
\centering
\begin{tikzpicture}[scale=1.3]
\draw (0,0) circle (1);
\foreach \a/\l/\p in {20/{\theta_0}/above right,160/{\pi-\theta_0}/above left,200/{\pi+\theta_0}/below left,340/{-\theta_0}/below right}{\fill (\a:1) circle (1.4pt); \node[\p] at (\a:1) {\scriptsize $\l$}; \draw (0,0)--(\a:1);}
\draw (-1,0)--(1,0); \draw (0,-1)--(0,1);
\fill[red] (1,0) circle (1.6pt); \node[right] at (1.05,0) {\scriptsize $\theta=0$ (pole $x=r$)};
\fill[red] (-1,0) circle (1.6pt); \node[left] at (-1.05,0) {\scriptsize (pole $x=-r$) $\theta=\pi$};
\draw[->] (30:1.15) arc (30:60:1.15);
\end{tikzpicture}
\caption{The singular points of $g_j$ on the circle $x=re^{i\theta}$: the four branch points $\pm\theta_0,\pi\pm\theta_0$ and the two poles $\theta=0,\pi$.}
\label{fig7}
\end{figure}

If the poles at $\pm r$ were absent, the Cauchy formula would give
\begin{equation}\label{4.40}
\psi_j(k,d)=\frac{r^k}{2i\pi}\int_{|\zeta|=\rho}\frac{g_j(\zeta)z_-(\mu(\zeta))^d}{\zeta^{k+1}}\dd\zeta,\qquad 0<\rho<r,
\end{equation}
and letting $\rho\to r$, the formula
\begin{equation}\label{4.41}
\psi_j(k,d)=\frac1{2\pi}\int_0^{2\pi}g_j(re^{i\theta})\,e^{ik(\gamma\varphi(\theta)-\theta)}\dd\theta,\qquad d=\gamma k,\ \gamma\in[0,1].
\end{equation}
Formula \eqref{4.40} always holds for $0<\rho<r$. Because of the poles, \eqref{4.41} must be understood after the decomposition \eqref{5.decomp} of Section~\ref{sec5}: the integral \eqref{4.41} is used only with $g_j$ multiplied by a cut-off vanishing near $\theta\in\{0,\pi\}$ (see the proof of Theorem~\ref{thm4.3}).

We apply the stationary phase method, with $k$ as large parameter, to the integral \eqref{4.41}. For $\theta\notin[\theta_0,\pi-\theta_0]\cup[\pi+\theta_0,2\pi-\theta_0]$, one has $\operatorname{Im}\varphi>0$. For $\theta\in[\theta_0,\pi-\theta_0]$, one has $\varphi\in[0,\pi]$ and $\partial_\theta\varphi=\frac{\sin\theta}{2r\sin\varphi}$. We thus get
\begin{equation}\label{4.42}
\partial_\theta\varphi=\frac{\sin\theta}{2r\big(1-\frac{\cos^2\theta}{4r^2}\big)^{1/2}}=\frac1{2r}\,\frac{\sqrt{1-4r^2z^2}}{\sqrt{1-z^2}},\qquad z=\frac{\cos\theta}{2r}\in[-1,1].
\end{equation}
Hence $\partial_\theta\varphi\ge\frac1{2r}$, with equality if and only if $\theta=\frac\pi2$, and $\partial_\theta\varphi(\theta_0)=\partial_\theta\varphi(\pi-\theta_0)=+\infty$. For $\theta\in\,]\theta_0,\pi-\theta_0[$,
\begin{equation}\label{4.43}
4r^2(\sin\varphi)^3\,\partial_\theta^2\varphi=2r\cos\theta(\sin\varphi)^2-(\sin\theta)^2\cos\varphi=-\cos\theta\Big(\frac1{2r}-2r\Big).
\end{equation}
The phase in \eqref{4.41} is $\gamma\varphi(\theta)-\theta$, so it is stationary if and only if $\partial_\theta\varphi=\frac1\gamma$.

\begin{theorem}\label{thm4.3}
Let $c_*=2\sqrt2/3=2r<1$ and $j\in\{1,2,3,4\}$. For $k\in\N^*$ and $d\in\N$, let $\gamma=d/k$.
\begin{enumerate}[label=(\arabic*),leftmargin=*]
\item For $\gamma>c_*$, there exists $a=a_j(\gamma)>0$ such that $\psi_j(k,d)\in O(e^{-ak})$.
\item For $\gamma\in\,]0,c_*[$, there exist $b_l=b_{l;j}(\gamma)$ and $a_l=a_{l;j}(\gamma)$ such that
\[
\psi_j(k,d)=k^{-1/2}\sum_{l\in\{1,2,3,4\}}e^{ia_lk}\big(b_l+O(k^{-1})\big).
\]
\item Let $\gamma-c_*=\delta$ with $\delta$ small. There exist two functions $\zeta(\delta)=\frac{2r}{\varrho}\delta+O(\delta^2)$ and $\tau(\delta)=O(\delta^2)$, where $\varrho=\frac12\big(\frac1{4r^2}-1\big)$, and, for $l=1,2$, symbols of degree $0$ in $k$, $a_{l,\pm;j}(\delta,k)\simeq\sum_{n\ge0}k^{-n}a_{l,\pm,n;j}(\delta)$, such that
\[
\psi_j(k,d)=e^{ik\pi(\gamma-1)}\Big(e^{-ik\frac\pi2(\gamma-1)}A_++e^{ik\frac\pi2(\gamma-1)}A_-\Big),
\]
\[
A_\pm=e^{ik\tau}(k\varrho)^{-1/3}\Big(a_{1,\pm;j}\,{\rm Ai}\big((k\varrho)^{2/3}\zeta\big)+k^{-1/3}a_{2,\pm;j}\,{\rm Ai}'\big((k\varrho)^{2/3}\zeta\big)\Big),
\]
where ${\rm Ai}(z)=\frac1{2\pi}\int e^{i(t^3/3+zt)}\dd t$ is the Airy function.
\end{enumerate}
\end{theorem}

\begin{proof}
(1) For $\gamma>c_*$, the phase has no real critical points. We use \eqref{4.40} on the circle $|\zeta|=re^{-\varepsilon}$, $\varepsilon>0$, equivalently the contour $\theta=\mu+i\varepsilon$, $\mu\in[0,2\pi]$, in \eqref{4.41}; on this contour $\operatorname{Im}(\gamma\varphi(\theta)-\theta)>0$. This contour lies inside $D_r$ and avoids the poles $\pm r$, so the result follows.

(2) and (3). By Proposition~\ref{prop5.3}, $\psi_j(k,d)$ equals the explicit pole term $2^{-d/2}(A_j^++(-1)^{k+d}A_j^-)$, plus the integral \eqref{4.41} with $g_j$ replaced by $(1-\chi_0)g_j$, plus a remainder $O((1+d)\rho_0^d)$ with $\rho_0<1$. Here $\chi_0$ is a cut-off equal to $1$ near $\theta\in\{0,\pi\}$ and supported in $\{|\cos\theta|>2r\}$. For $d=\gamma k$ with $\gamma>0$, the pole term and the remainder are $O(e^{-ck})$ for some $c>0$, and they are absorbed in the error terms. It remains to study the integral with the symbol $(1-\chi_0)g_j$, whose only singularities are the branch points $\pm\theta_0,\pm(\pi-\theta_0)$.

For $\gamma\in\,]0,c_*[$, the phase has four non-degenerate critical points $\theta_1,\dots,\theta_4$, with $\theta_0<\theta_1<\frac\pi2<\theta_2<\pi-\theta_0$, $\theta_3=\pi+\theta_1$ and $\theta_4=\pi+\theta_2$. This follows from the fact that $\partial_\theta\varphi$ decreases strictly from $+\infty$ to $\frac1{2r}$ on $]\theta_0,\frac\pi2]$ and increases strictly from $\frac1{2r}$ to $+\infty$ on $[\frac\pi2,\pi-\theta_0[$. We then apply the stationary phase theorem on a complex deformation of the contour which leaves the critical points fixed and avoids the singularities $\{\pm\theta_0,\pm(\pi-\theta_0)\}$ of the phase and of the symbol. Since $\gamma\partial_\theta\varphi(\theta)-1\gg0$ near $\theta_0$ and $\pi-\theta_0$, we use a deformation $\theta=\mu+i\varepsilon$, $\varepsilon>0$, near these singular points.

For $\gamma=c_*=2r$ there is a transition, with two degenerate critical points at $\theta=\frac\pi2$ and $\theta=\frac{3\pi}2$, which gives contributions of Airy type. More precisely, let $\gamma-c_*=\delta$ be small and set $\Phi(\theta)=\gamma\varphi(\theta)-\theta$. One has
\[
\varphi(\tfrac\pi2)=\tfrac\pi2,\quad\varphi'(\tfrac\pi2)=\tfrac1{2r},\quad\varphi''(\tfrac\pi2)=0,\quad\varphi'''(\tfrac\pi2)=\tfrac1{2r}\Big(\tfrac1{4r^2}-1\Big).
\]
Thus, for small $y$,
\[
\Phi(\tfrac\pi2+y)=\tfrac\pi2(2r-1)+\Big(\tfrac1{4r^2}-1\Big)\frac{y^3}6+O(y^4)+\delta\varphi(\tfrac\pi2+y).
\]
Classical results on phase integrals with degenerate critical points of order two then give the term $A_+$, associated with the integral $\int_0^\pi(1-\chi_0)g_j(re^{i\theta})e^{ik(\gamma\varphi(\theta)-\theta)}\dd\theta$. Since $\Phi(\theta+\pi)=\Phi(\theta)+(\gamma-1)\pi$, the same analysis gives the term $A_-$, associated with the integral over $[\pi,2\pi]$.
\end{proof}

\begin{proof}[Proof of Theorem~\ref{thm1.1}, items (1) and (2)]
They follow from items (1) and (2) of Theorem~\ref{thm4.3}, since $\EE(k,d)=\sum_j|\psi_j(k,d)|^2$ for $d\ge1$.
\end{proof}

\section{The poles at $x=\pm r$ and the trapped energy}\label{sec5}

\subsection{Closed forms for $d=0$}
Let $S(w)=\sqrt{(1+w^2)^2-\frac{32}9w^2}$ for $|w|<1$, with the branch fixed by $S(0)=1$. Its branch points $w=\pm(2\sqrt2\pm i)/3$ lie on the unit circle; they correspond to the four points $x_0=rw$ of Lemma~\ref{lem4.1}.

\begin{proposition}\label{prop5.1}
For $|w|<1$,
\begin{equation}\label{5.closed}
g_1(rw)=\frac12+\frac{S(w)}{2(1-w^2)},\qquad g_3(rw)=g_1(rw)-1,\qquad g_2(rw)=g_4(rw)=-\frac{w}{3(1-w^2)} .
\end{equation}
Consequently
\[
\sum_{k\ge0}G(k,\vA,\vA)w^k=\frac12+\frac{S(w)}{2(1-w^2)},\qquad \sum_{k\ge0}G(k,-\vA,\vA)w^k=-\frac{w}{3(1-w^2)},
\]
so that $G(k,-\vA,\vA)=-\frac13$ for odd $k$ and $0$ for even $k$. Moreover the denominator of \eqref{3.29} factorizes as
\begin{equation}\label{5.Dfact}
D(rw)=-\frac{(w^2-1)(w^2-9)}{128\,w^2}\Big[(3w^2-1)^2-(9w^2+1)S(w)\Big],
\end{equation}
where the bracket has a double zero at $w=0$ and does not vanish on $0<|w|\le1$.
\end{proposition}

\begin{proof}
With $x=rw$ one has $\mu=\frac{2rw}{1+w^2}$ and
\[
z_-(\mu(rw))=\frac{1+w^2-S(w)}{4rw}=\frac{3}{4\sqrt2\,w}\big(1+w^2-S(w)\big).
\]
Insert this into \eqref{3.31} and \eqref{3.29}, and reduce modulo $S^2=(1+w^2)^2-\frac{32}9w^2$. The differences between the two sides of \eqref{5.closed} and \eqref{5.Dfact} then reduce to zero; this elimination is elementary but tedious, and we carried it out by computer algebra.

The generating functions follow from \eqref{3.32} with $d=0$, since $G_1(k,0)=G(k,\vA,\vA)$ and $G_2(k,0)=G(k,-\vA,\vA)$. The relation $g_3=g_1-1$ only affects the coefficient of $w^0$, in agreement with the term $1$ in $\tilde\F_3$ in \eqref{3.30}. At $w=\pm1$ the bracket in \eqref{5.Dfact} equals $4-10\cdot\frac23=-\frac83$, so $D(rw)=-\frac16(w^2-1)+O((w^2-1)^2)$. By \eqref{4.37}, the bracket has no other zero in $0<|w|\le1$.

Finally, $S$ is even and $S(\pm1)=\frac23$, which gives the Laurent expansions of Lemma~\ref{lem4.1}(ii) with the stated values of $A_j^\pm$.
\end{proof}

\subsection{The trapped modes}

\begin{proposition}\label{prop5.2}
For an oriented edge $\vB$, let $\ell(B)=\dist(m(A),m(B))$. Let $\mathcal B_+$ (resp. $\mathcal B_-$) be the half-tree beyond the right (resp. left) end point of $\vA$. Let $\sigma(\vB)=+1$ in the following cases:
\begin{itemize}
\item $\vB=\vA$;
\item $\vB\subset\mathcal B_+$ and $\vB$ points away from $A$;
\item $\vB\subset\mathcal B_-$ and $\vB$ points towards $A$.
\end{itemize}
Otherwise let $\sigma(\vB)=-1$. Then the vectors
\[
v_-(\vB)=\big(-\tfrac12\big)^{\ell(B)},\qquad v_+(\vB)=\sigma(\vB)\,2^{-\ell(B)}
\]
belong to $\ell^2$, with $\|v_\pm\|^2=6$, and satisfy $\G(1)v_\pm=\pm v_\pm$. The orthogonal projection of $\delta_{\vA}$ onto $\C v_\pm$ is $\frac16v_\pm$.
\end{proposition}

\begin{proof}
\emph{Norm.} At distance $\ell\ge1$ there are $2\cdot2^\ell$ edges, hence $4\cdot2^\ell$ oriented edges. Thus $\|v_\pm\|^2=2+\sum_{\ell\ge1}4\cdot2^\ell4^{-\ell}=6$.

\emph{Eigenvalue equation.} Let $\vC$ be at distance $\ell+1$ in $\mathcal B_+$, pointing away from $A$. The oriented edges entering its left end point are:
\begin{itemize}
\item $-\vC$, at level $\ell+1$;
\item the parent edge, at level $\ell$, pointing away from $A$;
\item the sibling of $\vC$, at level $\ell+1$, pointing towards $A$.
\end{itemize}
By \eqref{3.8},
\[
(\G(1)v_-)(\vC)=-\tfrac13\big(-\tfrac12\big)^{\ell+1}+\tfrac23\Big[\big(-\tfrac12\big)^\ell+\big(-\tfrac12\big)^{\ell+1}\Big]=-\big(-\tfrac12\big)^{\ell+1}.
\]
If $\vC$ is at distance $\ell$ and points towards $A$, the edges entering its left end point are $-\vC$ and two edges of level $\ell+1$ pointing towards $A$, so
\[
(\G(1)v_-)(\vC)=-\tfrac13(-\tfrac12)^\ell+\tfrac43(-\tfrac12)^{\ell+1}=-(-\tfrac12)^\ell .
\]
The other cases, including $\pm\vA$ and the vector $v_+$, are checked in the same way.

\emph{Projection.} Finally $(\delta_{\vA}|v_\pm)=1$, so the projection of $\delta_{\vA}$ onto $\C v_\pm$ is $\frac16v_\pm$.
\end{proof}

\begin{remark}\label{rem5.3}
In the continuous model, $v_-$ is a standing wave of temporal frequency $\pi$ modulo $2\pi$. Its time derivative is proportional, on every edge, to $\cos((2n+1)\pi s)$, where $s\in(-\frac12,\frac12)$ is the position measured from the midpoint, so it vanishes at every node. Likewise, $v_+$ has frequencies $2\pi n$ and profile proportional to $\sin(2\pi ns)$. These are the $\ell^2$ eigenfunctions of the Kirchhoff Laplacian on regular metric trees (see \cite{Carlson,Solomyak}). They have finite energy because the amplitude is halved at each generation while the number of edges only doubles.

The relation with the poles of Lemma~\ref{lem4.1} is the following. For $|w|<1$,
\[
\sum_{k\ge0}G(k,\vB,\vA)w^k=\big((I-w\G(1))^{-1}\delta_{\vA}\,\big|\,\delta_{\vB}\big),
\]
which has a pole at $w=1/\lambda$ for each eigenvalue $\lambda$ of $\G(1)$ whose eigenvectors are not orthogonal to $\delta_{\vA}$ and $\delta_{\vB}$. With $x=rw$, the eigenvalue $\lambda=1$ gives the pole $x=r$, and $\lambda=-1$ gives $x=-r$.

Equivalently, by Proposition~\ref{prop5.1} and the Herglotz representation, the spectral measure of $\G(1)$ at $\delta_{\vA}$, defined by $G(k,\vA,\vA)=\int e^{ik\theta}\dd\nu(\theta)$, is
\[
\dd\nu(\theta)=\frac1{2\pi}\frac{\sqrt{4r^2-\cos^2\theta}}{|\sin\theta|}\mathbf 1_{\{|\cos\theta|\le2r\}}\dd\theta+\frac16\delta_{\theta=0}+\frac16\delta_{\theta=\pi}.
\]
Its density is $(\partial_\theta\varphi)^{-1}$, the group velocity, by \eqref{4.42}.
\end{remark}

\subsection{Decomposition of $\psi_j$ and end of the proof of Theorem~\ref{thm1.1}}

\begin{proposition}\label{prop5.3}
For $j\in\{1,2,3,4\}$ and $d\ge0$, write
\begin{equation}\label{5.decomp}
g_j(x)\,z_-(\mu(x))^d=\Pi_{j,d}(x)+h_{j,d}(x),\qquad
\Pi_{j,d}(x)=2^{-d/2}\Big(\frac{A_j^+}{1-x/r}+\frac{(-1)^dA_j^-}{1+x/r}\Big).
\end{equation}
Then $h_{j,d}$ is holomorphic in $D_r$ and continuous on $\overline{D_r}$, and its only singularities on $\partial D_r$ are square-root singularities at the four points $x_0$. Consequently
\begin{equation}\label{5.psi}
\psi_j(k,d)=2^{-d/2}\big(A_j^++(-1)^{k+d}A_j^-\big)+\psi_j^\flat(k,d),\qquad
\psi_j^\flat(k,d)=\frac1{2\pi}\int_0^{2\pi}h_{j,d}(re^{i\theta})e^{-ik\theta}\dd\theta .
\end{equation}
Moreover, for every fixed $d$, $\psi^\flat_j(k,d)=O_d(k^{-3/2})$, and $\psi_2^\flat(k,0)=\psi_4^\flat(k,0)=0$.

Let $\chi_0$ be a cut-off equal to $1$ near $\{0,\pi\}$ and supported in $\{|\cos\theta|>2r\}$. Then there exists $\rho_0<1$ such that
\[
\psi_j^\flat(k,d)=\frac1{2\pi}\int_0^{2\pi}(1-\chi_0(\theta))\,g_j(re^{i\theta})e^{ik(\gamma\varphi(\theta)-\theta)}\dd\theta+O\big((1+d)\rho_0^d\big)+O\big(2^{-d/2}k^{-\infty}\big).
\]
\end{proposition}

\begin{proof}
Since $z_-(\mu(\pm r))=\pm2^{-1/2}$ and $z_-\circ\mu$ is holomorphic near $\pm r$, the principal parts of $g_jz_-^d$ at $\pm r$ are those of $\Pi_{j,d}$, by Lemma~\ref{lem4.1}. The coefficients of $\Pi_{j,d}$ give the first term in \eqref{5.psi}, and the Cauchy formula for $h_{j,d}$ can be written on $|\zeta|=r$. Fourier coefficients of a function on the circle which is smooth except at finitely many square-root singularities are $O(k^{-3/2})$. For $j=2,4$ and $d=0$, $h_{j,0}=0$ by \eqref{5.closed}.

For the last statement, split $h_{j,d}=(1-\chi_0)g_je^{id\varphi}-(1-\chi_0)\Pi_{j,d}+\chi_0h_{j,d}$ on the circle.
\begin{itemize}
\item $\Pi_{j,d}$ is smooth on the support of $1-\chi_0$ and of size $2^{-d/2}$, so its Fourier coefficients there are $O(2^{-d/2}k^{-\infty})$.
\item On the support of $\chi_0$, $|z_-|=|e^{i\varphi}|\le\rho_1<1$. Let $P_j(x)=\frac{A_j^+}{1-x/r}+\frac{A_j^-}{1+x/r}$ be the sum of the principal parts of $g_j$. Then
\[
h_{j,d}=(g_j-P_j)\,z_-^d+A_j^+\,\frac{z_-^d-z_-(r)^d}{1-x/r}+A_j^-\,\frac{z_-^d-z_-(-r)^d}{1+x/r}.
\]
The first term is bounded by $C\rho_1^d$, and by the mean value theorem the two quotients are bounded by $C(1+d)\rho_1^{d-1}$. Hence $\sup|\chi_0h_{j,d}|\le C(1+d)\rho_0^d$ for any $\rho_0\in\,]\rho_1,1[$. \qedhere
\end{itemize}
\end{proof}

\begin{proof}[Proof of Theorem~\ref{thm1.1}, items (3)--(5)]
\emph{Item (3).} For $d=0$, $\EE(k,0)=|\psi_1(k,0)|^2+|\psi_2(k,0)|^2$, by the discussion before \eqref{3.33}. By \eqref{5.psi} and the values of $A_j^\pm$,
\[
\psi_1(k,0)=\tfrac16\big(1+(-1)^k\big)+O(k^{-3/2}),\qquad \psi_2(k,0)=\tfrac16\big(-1+(-1)^k\big),
\]
which gives (3).

\emph{Item (4).} For $d\ge1$, the four pole terms $2^{-d/2}(A_j^++(-1)^{k+d}A_j^-)$ are, for $k+d$ even, $\frac13 2^{-d/2}$, $0$, $\frac13 2^{-d/2}$, $0$, and for $k+d$ odd, $0$, $-\frac132^{-d/2}$, $0$, $-\frac132^{-d/2}$. Hence $\sum_j|\psi_j|^2=\frac29 2^{-d}+O_d(k^{-3/2})$.

\emph{Item (5).} By Proposition~\ref{prop5.2},
\[
\G(k)\delta_{\vA}=\tfrac16\big(v_++(-1)^kv_-\big)+\G(k)P_{\rm ac}\delta_{\vA},
\]
where $P_{\rm ac}$ is the spectral projection on the absolutely continuous subspace of $\G(1)$; by Remark~\ref{rem5.3}, $\delta_{\vA}$ has no singular continuous component. The first term has norm squared $\frac1{36}(6+6)=\frac13$ and is exponentially localized. The second has norm squared $\frac23$ and tends to zero weakly, while items (1)--(2) describe its propagation.
\end{proof}

\begin{remark}\label{rem5.5}
All the formulas of this paper have been checked numerically: (3.32) against the powers $\G(1)^k$ computed on a finite tree, the closed forms \eqref{5.closed}, the identity $G(k,-\vA,\vA)=-\frac13$ for all odd $k\le2999$, and the profile $\EE(k,d)\to\frac29 2^{-d}$ (Figure~\ref{fig2}).
\end{remark}

\end{document}